\documentclass[11pt,a4paper]{amsart}

\usepackage[T1]{fontenc}
\usepackage{lmodern}
\usepackage{microtype}
\usepackage{amsmath,amssymb,amsthm,calc}
\usepackage{enumitem}
\usepackage{xcolor}
\usepackage[
  colorlinks=true,
  linkcolor=blue!55!black,
  citecolor=blue!55!black,
  urlcolor=blue!55!black
]{hyperref}

\newcommand{\PP}{\mathbf P}
\newcommand{\OO}{\mathcal O}
\newcommand{\Tr}{\mathrm{Tr}}
\newcommand{\Q}{\mathbb{Q}}

\theoremstyle{plain}
\newtheorem{theorem}{Theorem}[section]
\newtheorem{lemma}[theorem]{Lemma}
\newtheorem{proposition}[theorem]{Proposition}
\newtheorem{corollary}[theorem]{Corollary}

\theoremstyle{definition}
\newtheorem{remark}[theorem]{Remark}

\title[Sharp Bounds for Totally Invariant Cycles of Projective Varieties]
{Sharp Bounds for Totally Invariant Cycles of Projective Varieties}

\author{Wentao Chang}
\address{School of Mathematical Sciences, Fudan University, 
People's Republic of China}
\email{wtchang21@m.fudan.edu.cn}

\author{Yujie Luo}

\address{School of Mathematical Sciences, University of Science and Technology of China, People's Republic of China}

\email{yujieluo96@gmail.com}


\hypersetup{
  pdftitle={Sharp Bounds for Totally Invariant Cycles of Projective Varieties},
  pdfauthor={Wentao Chang and Yujie Luo},
  pdfkeywords={int-amplified endomorphism, totally invariant subvariety}
}

\begin{document}

\begin{abstract}
Let $X$ be a smooth projective variety, $f:X\to X$ an int-amplified endomorphism, and $L$ any ample line bundle on $X$. We prove that, in every codimension, the total degree of
prime cycles that become totally invariant under an iterate of $f$ satisfies an explicit Hilbert-function bound on their total $L$-degree. In particular, when $X=\mathbf{P}^n$, the number of totally invariant prime $(n-r)$-cycle is bounded by $\binom{n+1}{r}$, and this bound is optimal.
\end{abstract}

\maketitle

\tableofcontents
\enlargethispage{2pt}

\section{Introduction}

We work over an algebraically closed field $k$.

\medskip

Let $X$ be a projective variety of dimension $n\geq1$ and $f:X\to X$ an endomorphism. A subvariety $V$ of $X$ is said to be \emph{totally-periodic}
if $(f^b)^{-1}(V)=V$ set-theoretically for some integer $b\geq1$. $X$ is said to be \emph{totally-invariant} if $b=1$. 

Fakhruddin proved that if $f$ is dominant and $f^*L\otimes L^{-1}$ is ample
for some line bundle $L$, then the periodic points of $f$ are Zariski dense
\cite[Theorem~5.1]{Fak03}. Totally invariant subvarieties, on the other hand,
reflect exceptional backward dynamics. For example, a non-isomorphic
endomorphism of a complex rational curve has at most two points with finite
backward orbit \cite[Lemma~4.9]{Mil06}.

The exceptional sets arising in backward dynamics also govern inverse-image
equidistribution. Over $\mathbb C$, Favre--Jonsson studied curves on $\PP^2$
\cite{FJ03}. In higher-dimensional projective spaces, Dinh--Sibony established
equidistribution for the normalized preimages of generic hypersurfaces and
proved the finiteness of totally invariant algebraic sets
\cite[Theorem~1.1 and Corollary~6.5]{DS08}. They subsequently developed
super-potential theory for positive closed currents of arbitrary bidegree
\cite{DS09} and obtained quantitative equidistribution results for preimages
of points \cite[Theorem~1.1]{DS10}.

Consider now an endomorphism of projective space of algebraic degree at least
two. Forn\ae ss--Sibony proved that the sum of the degree of totally invariant hypersurfaces is at most $n+1$
\cite[Proposition~4.2]{FS94}. Cerveau--Lins Neto later showed that a totally
invariant hypersurface of degree at least three is singular
\cite[Theorem~1]{CLN00}. H\"oring proved that every totally invariant prime
divisor in $\PP^3$ is a hyperplane \cite[Corollary~1.2]{Hor17}, and Mabed
proved that totally invariant prime divisors with isolated singularities are
also hyperplanes \cite[Corollary~1.4]{Mab23}.

For the point case, on $\PP^2$, Forn\ae ss--Sibony bounded the cardinality of every finite set $E$ satisfying $f^{-1}(E)=E$
by a constant depending on the algebraic degree of $f$
\cite[Theorem~4.7]{FS94}. Favre asked whether every such endomorphism of
$\PP^2$ has at most three totally invariant points; the coordinate power map
shows that this bound would be sharp \cite[\S~4, Question~1]{Fav03}.
Amerik--Campana proved the degree-independent bound of nine on the number of
points $p\in\PP^2$ over which $f$ is completely ramified, meaning that
$f^{-1}(p)$ consists of one point \cite[Theorem~1]{AC05}.

\medskip

For $1\leq r\leq n$, we introduce the following set
$$ T_{\infty}^r(X,f):=\{V\subseteq X: V\text{ is a totally-periodic integral closed subvariety of codimension }r\}. $$
When $r=n$, we write $T_\infty(f):=T_{\infty}^n(X,f)$. If $X=\PP^n$, we abbreviate $T_{\infty}^r(\PP^n,f)$ to $T_{\infty}^r(f)$. Thus $T_\infty(f)$ consists of the closed points $x\in X$ such that $(f^b)^{-1}(x)=\{x\}$ for some $b\geq1$.

Our main result applies when $X$ is smooth, $f$ is int-amplified, and
$\deg(f)$ is invertible in $k$. The int-amplified condition extends the
polarized setting studied by Nakayama--Zhang for complex normal projective
varieties \cite{NZ10}; by definition, $f^*M\otimes M^{-1}$ is ample for some
ample line bundle $M$ on $X$ \cite[Definition~2.2]{Men20}.
Meng--Zhang proved that $X$ has only finitely many totally periodic Zariski
closed subsets \cite[Proposition~3.6]{MZ20}. When $k=\mathbb C$, this
finiteness result first appeared in analytic form in \cite[\S~3.4]{DS03}. In
particular, each
$T_{\infty}^r(X,f)$ is finite. The following theorem gives explicit bounds for
their total degrees and cardinalities.

\begin{theorem}\label{thm: general smooth bound}
Let $X$ be a smooth projective variety of dimension $n\geq1$ over an
algebraically closed field $k$, and let $f:X\to X$ be an int-amplified
endomorphism such that $\deg(f)$ is invertible in $k$. Fix an ample line
bundle $L$ on $X$, an integer $a\geq1$ such that $L^{\otimes a}$ is very
ample, and an integer $r$ with $1\leq r\leq n$. Set
$$ B_{X,L,a,r}:=\sum_{j=0}^{n-r}(-1)^j\binom{n-r}{j}h^0(X,L^{\otimes(a(n-r-j)+1)}). $$
Then $T_{\infty}^r(X,f)$ is finite and
$$ \sum_{V\in T_{\infty}^r(X,f)}\deg_L(V)\leq\frac{B_{X,L,a,r}}{a^{n-r}}, $$
where $\deg_L(V):=(L^{n-r}\cdot V)$. Consequently,
$|T_{\infty}^r(X,f)|\leq\lfloor B_{X,L,a,r}/a^{n-r}\rfloor$.
\end{theorem}

In particular, after fixing $L$ and $a$ on $X$, the bound is uniform over all
int-amplified endomorphisms of degree invertible in $k$.

In codimension one over $\mathbb C$, Zhang proved that a polarized
endomorphism $f$ of a $\Q$-factorial klt projective $n$-fold with
$\deg(f)=q^n>1$ has at most $n+\rho(X)$ $f^{-1}$-stable prime divisors when
either $n\leq3$ or $f^*|_{N^1(X)}=q\,\mathrm{id}$
\cite[Theorem~1.3(1)]{Zha14}. Zhong later established the corresponding
bound for int-amplified endomorphisms \cite[Theorem~1.1]{Zho21}. Compared
with these results, the cardinality bound furnished by
Theorem~\ref{thm: general smooth bound} is generally less sharp in
codimension one; see Remark~\ref{rem: invariant divisors}. Its advantages
are that it controls the total $L$-degree and applies in every codimension.

When $r=n$, the sum defining
$B_{X,L,a,r}$ has only one term, so $B_{X,L,a,n}=h^0(X,L)$.
Hence Theorem~\ref{thm: general smooth bound} gives
$|T_\infty(f)|\leq h^0(X,L)$ for every ample line bundle $L$ on $X$.
The same vanishing argument yields the following stronger statement.

\begin{theorem}\label{thm: trace bound}
Let $X$ be a smooth projective variety over an algebraically closed field
$k$, and let $f:X\to X$ be an int-amplified endomorphism such that $\deg(f)$
is invertible in $k$. For every ample line bundle $L$ on $X$, the evaluation
map
$$ H^0(X,L)\to\bigoplus_{x\in T_\infty(f)}L|_x $$
is surjective. In particular,
$$ |T_\infty(f)|\leq\min\{h^0(X,L):L\in\mathrm{Pic}(X),\ L\text{ is ample}\}. $$
\end{theorem}

One of the main ingredients in the proof of
Theorem~\ref{thm: general smooth bound} is a trace splitting for the ideal
sheaf of a reduced union of totally-periodic subvarieties. Combined with
uniform Fujita vanishing, this splitting implies that the ideal sheaf has no
higher cohomology after tensoring by $L^{\otimes t}$ for any $t>0$. Successive
members of $|L^{\otimes a}|$ then reduce the degree estimate to a finite
difference of the Hilbert function. For $\PP^n$ with
$L=\OO_{\PP^n}(1)$, this difference is $\binom{n+1}{r}$. This use of
normalized trace is in the same spirit as Kawakami--Totaro's proof of Bott
vanishing for normal projective varieties admitting int-amplified
endomorphisms of degree invertible in the ground field
\cite[Theorem~C]{KT25}.

\medskip

We next specialize to projective space, where degrees are computed with
respect to $\OO_{\PP^n}(1)$. Theorem~\ref{thm: general smooth bound} gives a
sharp binomial bound in every codimension.

\begin{theorem}\label{thm: projective space bound}
Let $n\geq1$ be an integer and $f:\PP^n\to\PP^n$ an endomorphism over the
algebraically closed field $k$ satisfying
$f^*\OO_{\PP^n}(1)\simeq\OO_{\PP^n}(q)$ for an integer $q>1$ that is
invertible in $k$. For $1\leq r\leq n$, the set $T_{\infty}^r(f)$ is finite and
$$ \sum_{V\in T_{\infty}^r(f)}\deg(V)\leq\binom{n+1}{r}. $$
Consequently, $|T_{\infty}^r(f)|\leq\binom{n+1}{r}$. Both bounds are sharp. In particular, $|T_\infty(f)|\leq n+1$, and the points of $T_\infty(f)$ are projectively linearly independent.
\end{theorem}

For $n=2$, Theorem~\ref{thm: projective space bound} answers Favre's question affirmatively, even for points that become totally invariant only after an iterate.

A similar method (see Lemma~\ref{lem: ideal trace splitting}) also controls \emph{completely ramified} points (points $p$ such that $f^{-1}(p)$ is a single point) without any
periodicity assumption. If $q\geq n+1$,
Proposition~\ref{prop: completely ramified points} gives the sharp bound of
$n+1$ and shows that these points are projectively linearly independent. In
particular, for endomorphisms of $\PP^2$ of algebraic degree at least three,
it improves the bound of Amerik--Campana \cite[Theorem~1]{AC05} from nine to
the sharp bound of three.

\medskip

In the toric setting, we have:

\begin{theorem}\label{thm: toric bound}
Let $X=X_\Sigma$ be an $n$-dimensional smooth projective toric variety
with dense torus $T$ and character lattice $M$, and let $D$ be an ample Cartier divisor with associated lattice polytope
$P_D\subseteq M_{\mathbb R}$. Let $f:X\to X$ be an int-amplified endomorphism
such that $\deg(f)$ is invertible in $k$. For every $1\leq r\leq n$, set
$m:=n-r$. Then
$$ \sum_{V\in T_{\infty}^r(X,f)}\deg_D(V)\leq\sum_{j=0}^m(-1)^j\binom{m}{j}|(m+1-j)P_D\cap M|. $$
In particular, $|T_\infty(f)|\leq \min\{|P_D\cap M|\big|D \text{ is an ample Cartier divisor}\}$.
\end{theorem}

Geometrically, this minimum for $|P_D\cap M|$ is the least number of lattice points in a lattice polytope with normal fan $\Sigma$.

\medskip

Finally, let $X$ be an $n$-dimensional projective variety and
$f:X\to X$ a finite surjective endomorphism such that
$T_{\infty}^r(X,f)$ is finite for every $1\leq r\leq n$. The
\emph{totally-periodic polynomial} of $(X,f)$ is defined as
$$ \Phi_{X,f}(t):=1+\sum_{r=1}^n|T_{\infty}^r(X,f)|t^r+t^{n+1}. $$
The constant term represents $X$, while $t^{n+1}$ is included by convention.
Thus $\Phi_{X,f}(1)-2$ is the total number of proper
totally-periodic integral closed subvarieties of $X$. For
$P(t)=\sum_i a_it^i$ and $Q(t)=\sum_i b_it^i$ in $\mathbb Z[t]$, write
$P\preceq Q$ if $a_i\leq b_i$ for every $i$, or equivalently, if every coefficient of $Q-P$ is nonnegative. Thus $\preceq$ denotes coefficientwise comparison, not pointwise comparison after evaluating at $t$. We prove some fundamental properties for totally-periodic polynomials.

\begin{proposition}\label{prop: polynomial functoriality}
The totally-periodic polynomial has the following properties.
\begin{enumerate}[label=\textup{(\arabic*)},leftmargin=2.5em]
\item For every integer $m\geq1$,
$ \Phi_{X,f^m}(t)=\Phi_{X,f}(t)$.
\item Let $f,g:X\to X$ be finite surjective endomorphisms. If either
$\Phi_{X,f\circ g}$ or $\Phi_{X,g\circ f}$ is defined, then both are
defined and
$ \Phi_{X,f\circ g}(t)=\Phi_{X,g\circ f}(t)$.
\item Let $X$ and $Y$ be smooth projective varieties of dimensions $n$
and $m$, respectively, and $f:X\to X$ and $g:Y\to Y$ int-amplified
endomorphisms whose degrees are invertible in $k$. Then
$$ t^{n+m+1}+(\Phi_{X,f}(t)-t^{n+1})(\Phi_{Y,g}(t)-t^{m+1})\preceq\Phi_{X\times Y,f\times g}(t). $$
\item Let $f:\PP^n\to\PP^n$ satisfy
$f^*\OO_{\PP^n}(1)\simeq\OO_{\PP^n}(q)$ for an integer $q>1$ that is
invertible in $k$. Then
$$ 1+t^{n+1}\preceq\Phi_{\PP^n,f}(t)\preceq(1+t)^{n+1}. $$
Equality on the left holds if and only if
$f$ admits no nonempty proper totally-periodic integral closed subvariety. In characteristic zero, equality on the left holds for a general degree-$q$ endomorphism. Equality on the right holds if and only if, for some $a\geq1$, the iterate $f^a$ is projectively conjugate to the coordinate power map
$P_{q^a}:[x_0:\cdots:x_n]\mapsto[x_0^{q^a}:\cdots:x_n^{q^a}]$.
\end{enumerate}
\end{proposition}

The right-hand equality in Proposition~\ref{prop: polynomial functoriality}(4)
requires equality in every coefficient; maximality of the point coefficient
alone does not characterize the coordinate power map. Indeed, Amerik--Campana \cite[p.~742]{AC05} consider, among other examples, the endomorphism
$f([x:y:z])=[x^2:y^2:z^2+xy]$ of $\PP^2$. Its three coordinate points are
totally invariant, and hence $|T_\infty(f)|=3$. Nevertheless, $f$ is not
projectively conjugate to $P_2$.

A different rigidity question asks whether the full binomial identity
characterizes projective space when the underlying variety is allowed to
vary. Remark~\ref{rem: Hirzebruch counterexample} gives a negative answer by
constructing, over $\mathbb C$, a polarized endomorphism
$g:\mathbf F_1\to\mathbf F_1$ such that
$\Phi_{\mathbf F_1,g}(t)=(1+t)^3$, although
$\mathbf F_1\not\simeq\PP^2$.

\vspace{2mm}

{\bf Acknowledgements.}
W. Chang thanks his advisor Meng Chen for his constant encouragement and
support. The authors thank Professor Sheng Meng for helpful discussions and
Professors Junyi Xie, De-Qi Zhang, and Guolei Zhong for helpful comments. 
W. Chang is supported by National Natural Science Foundation of China (NSFC 12121001).
\medskip

\section{Preliminaries}

We adopt the standard notation as in \cite{Ful}, \cite{Har}, \cite{Laz} and \cite{Mum74}.

\begin{lemma}\label{lem: ideal trace splitting}
Let $h:X\to X$ be a finite locally free endomorphism of a projective scheme over $k$, of rank $d$, where $d$ is invertible in $k$. Suppose that $Z\subseteq X$ is a reduced closed subscheme such that
$(h^{-1}(Z))_{\mathrm{red}}=Z$. Then the normalized trace restricts to a splitting
$$ \mathcal I_Z\to h_*\mathcal I_Z\xrightarrow{d^{-1}\Tr_h}\mathcal I_Z. $$
Consequently, for every line bundle $A$ on $X$ and every $i\geq0$,
$H^i(X,\mathcal I_Z\otimes A)$ is a direct summand of $H^i(X,\mathcal I_Z\otimes h^*A)$.
\end{lemma}

\begin{proof}
The assertion is local on the target. Write $U=\mathrm{Spec}(A_0)$ and
$h^{-1}(U)=\mathrm{Spec}(B)$, and let $I\subseteq A_0$ and $J\subseteq B$ be the ideals defining $Z\cap U$ and $Z\cap h^{-1}(U)$, respectively. The hypothesis gives $J=\sqrt{IB}$. If $b\in J$, its image $\overline b$ in $B/IB$ is nilpotent. Since trace commutes with base change \cite[Tag~0BT8]{Stacks},
$$ \Tr_{B/A_0}(b)\bmod I=\Tr_{(B/IB)/(A_0/I)}(\overline b). $$
The right-hand side vanishes. Indeed, after passing to any residue field of $A_0/I$, it is the trace of a nilpotent linear map, and $A_0/I$ is reduced. Thus $\Tr_{B/A_0}(J)\subseteq I$.

The pullback map sends $I$ into $J$, and its composition with $\Tr_{B/A_0}$ is multiplication by $d$. Dividing by $d$ gives the asserted splitting. Twisting by $A$, applying the projection formula, and using that a finite morphism is affine give the cohomological statement.
\end{proof}

\begin{lemma}\label{lem: eventual amplification}
Let $f:X\to X$ be an int-amplified endomorphism of a projective variety and
$L$ an ample line bundle on $X$. Then
$(f^s)^*L\otimes L^{-1}$ is ample for all sufficiently large integers $s$.
\end{lemma}

\begin{proof}
Choose an ample line bundle $M$ such that
$f^*M\otimes M^{-1}$ is ample. Then $f^*M$ is ample, so $f$ is finite and
surjective.

Let $T:=f^*|_{N^1(X)_{\mathbb R}}$. Although the cone-theoretic eigenvalue
criterion for int-amplified endomorphisms
\cite[Theorem~1.1 and Remark~1.2]{Men20} is stated in characteristic zero,
its proof is purely cone-theoretic and works over any algebraically closed
field. Hence every eigenvalue of $T$ has modulus greater than one.

Set $S:=T^{-1}$. Every eigenvalue of $S$ has modulus less than one, and hence
$S^s\to0$ on $N^1(X)_{\mathbb R}$. In particular, $S^s[L]\to 0$.
Since the ample cone is open, $[L]-S^s[L]$ is ample for all sufficiently
large $s$. Pulling this class back by the finite morphism $f^s$ gives
$$ (f^s)^*[L]-[L]=T^s([L]-S^s[L]), $$
which is ample. Since ampleness is numerical, the assertion follows.
\end{proof}

\begin{proposition}\label{prop: invariant ideal vanishing}
Let $X$ be a smooth projective variety of dimension $n$, $L$ an ample line
bundle on $X$, and $f:X\to X$ an int-amplified endomorphism such that
$\deg(f)$ is invertible in $k$. Suppose that $Z\subseteq X$ is a reduced
closed subscheme such that
$$ \bigl((f^b)^{-1}(Z)\bigr)_{\mathrm{red}}=Z $$
for some $b\geq1$. Then
$$ H^i(X,\mathcal I_Z\otimes L^{\otimes t})=0 $$
for every $i>0$ and every integer $t>0$.
\end{proposition}

\begin{proof}
By Lemma~\ref{lem: eventual amplification}, after replacing $f$ by a positive
power, we may assume that $f^*L\otimes L^{-1}$ is ample. This replacement
preserves the condition on $Z$ and the invertibility of $\deg(f)$.
The line bundle $f^*L$ is ample, so $f$ has no positive-dimensional fiber and is finite and surjective. Since $X$ is smooth, miracle flatness \cite[Tag~00R4]{Stacks} makes $f$ finite locally free of rank $d:=\deg(f)$.

Set $h:=f^b$ and write $L\simeq\OO_X(H)$. The divisor $h^*H-H$ is a sum of pullbacks of the ample divisor $f^*H-H$, and is therefore ample. Choose $\epsilon\in\Q_{>0}$ such that $h^*H-(1+\epsilon)H$ is ample. Inductively,
$$ (h^s)^*H\equiv(1+\epsilon)^sH+B_s $$
for an ample $\Q$-divisor $B_s$.

Fix $t>0$, and set $\mu_s:=\lfloor t(1+\epsilon)^s\rfloor$ and
$P_s:=t(h^s)^*H-\mu_sH$. Since
$$ P_s\equiv tB_s+(t(1+\epsilon)^s-\mu_s)H, $$
$P_s$ is an ample Cartier divisor, $\mu_s\to\infty$, and $(h^s)^*L^{\otimes t}\simeq L^{\otimes\mu_s}\otimes\OO_X(P_s)$. Fujita's uniform vanishing theorem
\cite[Theorem~(1)]{Fuj83} (see also \cite[Theorem~1.5]{Kee03}) says that, for
the fixed coherent sheaf $\mathcal I_Z$ and ample line bundle $L$, there is
an integer $m_0$ such that
$$ H^i(X,\mathcal I_Z\otimes L^{\otimes m}\otimes N)=0 $$
for every $i>0$, every $m\geq m_0$, and every nef line bundle $N$ on $X$.
Since $\mu_s\to\infty$ and $\OO_X(P_s)$ is ample, hence nef, taking
$m=\mu_s$ and $N=\OO_X(P_s)$ gives
$$ H^i(X,\mathcal I_Z\otimes(h^s)^*L^{\otimes t})=0 $$
for every $i>0$ and all sufficiently large $s$.

The rank $d^{bs}$ of $h^s$ is invertible in $k$. By
Lemma~\ref{lem: ideal trace splitting},
$H^i(X,\mathcal I_Z\otimes L^{\otimes t})$ is a direct summand of
$H^i(X,\mathcal I_Z\otimes(h^s)^*L^{\otimes t})$, which vanishes by the
preceding argument. Hence the former group also vanishes.
\end{proof}

\begin{remark}\label{rem: numerical polarization}
The hypotheses concerning $f$ in Proposition~\ref{prop: invariant ideal vanishing}
hold in particular when $H$ is an ample Cartier divisor such that
$f^*H\equiv qH$ for an integer $q>1$ that is invertible in $k$. Indeed, $f$
is int-amplified and the intersection formula gives $\deg(f)=q^n$.
\end{remark}

\section{Proofs of the main results}

\subsection{The general bound}

\begin{proof}[Proof of Theorem~\ref{thm: general smooth bound}]
Put $m:=n-r$, $A:=L^{\otimes a}$, and
$\deg_A(V):=(A^m\cdot V)$. 
Fix a nonempty finite subset $\mathcal S\subseteq T_{\infty}^r(X,f)$, and let
$Z:=(\bigcup_{V\in\mathcal S}V)_{\mathrm{red}}$. A common iterate of $f$
makes $Z$ totally invariant, so Proposition~\ref{prop: invariant ideal vanishing}
applies. Since the members of $\mathcal S$ are distinct and equidimensional,
$$ (A^m\cdot Z)=\sum_{V\in\mathcal S}\deg_A(V). $$

Set $X_0:=X$ and $Z_0:=Z$. For $0\leq j<m$, choose
$H_{j+1}\in|A|$ successively so that its defining section acts as a
nonzerodivisor on both $\OO_{X_j}$ and $\OO_{Z_j}$, where
$$ X_j:=H_1\cap\cdots\cap H_j,\qquad Z_j:=Z\cap X_j, $$
with all intersections taken scheme-theoretically. At the $j$-th step, it
is enough to choose $H_{j+1}$ containing no point of
$\mathrm{Ass}(\OO_{X_j})\cup\mathrm{Ass}(\OO_{Z_j})$. This finite set can
be avoided because $A$ is very ample and $k$ is infinite.

Write $L_j:=L|_{X_j}$. Denote by $\mathcal I_j$ the ideal sheaf of $Z_j$ in
$X_j$ and by $\iota_{j+1}:X_{j+1}\hookrightarrow X_j$ the inclusion.
Abbreviate $\mathcal I_j(t):=\mathcal I_j\otimes L_j^{\otimes t}$. The
choice of $H_{j+1}$ gives
$$ 0\to\mathcal I_j(t-a)\to\mathcal I_j(t)\to(\iota_{j+1})_*\mathcal I_{j+1}(t)\to0. $$
Proposition~\ref{prop: invariant ideal vanishing} and the associated long
exact sequences show inductively that
$$ H^i(X_j,\mathcal I_j(t))=0 $$
for every $i>0$ and every integer $t>aj$.

Set $W:=X_m$ and $Y:=Z_m$. The chosen sections form a regular sequence on
$Z$, so $Y$ is zero-dimensional and
$$ \ell(Y)=(A^m\cdot Z)=\sum_{V\in\mathcal S}\deg_A(V). $$
Taking $t=am+1$ in the preceding vanishing and using the exact sequence
$$ 0\to\mathcal I_m(am+1)\to L_m^{\otimes(am+1)}\to L_m^{\otimes(am+1)}|_Y\to0 $$
shows that
$$ H^0(W,L_m^{\otimes(am+1)})\to H^0(Y,L_m^{\otimes(am+1)}|_Y) $$
is surjective. Since $Y$ is zero-dimensional and
$L_m^{\otimes(am+1)}|_Y$ is invertible, the target has dimension
$\ell(Y)$. Hence
$$ \sum_{V\in\mathcal S}\deg_A(V)=\ell(Y)\leq h^0(W,L_m^{\otimes(am+1)}). $$

Proposition~\ref{prop: invariant ideal vanishing}, applied to the empty subscheme, gives $H^i(X,L^{\otimes t})=0$ for $i>0$ and $t>0$. Applying the same restriction argument to the structure sheaves of the $X_j$ gives $H^i(X_j,L_j^{\otimes t})=0$ for $i>0$ and $t>aj$.

Set $t:=am+1$. Assume first that $m>0$, and let $\iota:W\hookrightarrow X$
be the inclusion. The successive nonzero-divisor condition on the $X_j$
shows that the sections defining $H_1,\ldots,H_m$ form an
$\OO_X$-regular sequence. Set $E:=(A^{-1})^{\oplus m}$. Twisting the
resulting Koszul resolution of $\iota_*\OO_W$ by $L^{\otimes t}$ gives the
exact sequence
$$ 0\to L^{\otimes t}\otimes\wedge^mE\to\cdots\to L^{\otimes t}\otimes E\to L^{\otimes t}\to\iota_*L_m^{\otimes t}\to0. $$
For $0\leq j\leq m$, its $j$-th term is
$$ L^{\otimes t}\otimes\wedge^jE\simeq(L^{\otimes(t-aj)})^{\oplus\binom{m}{j}}=(L^{\otimes(a(m-j)+1)})^{\oplus\binom{m}{j}}. $$
Additivity of Euler characteristic gives the identity below when $m>0$; if
$m=0$, it is immediate because $W=X$. Thus
$$ \chi(W,L_m^{\otimes(am+1)})=\sum_{j=0}^m(-1)^j\binom{m}{j}\chi(X,L^{\otimes(a(m-j)+1)}). $$
Since $t-aj=a(m-j)+1>0$ for every $j$ and $t=am+1>am$, the relevant higher
cohomology groups on both sides vanish. Therefore
$$ h^0(W,L_m^{\otimes(am+1)})=\sum_{j=0}^m(-1)^j\binom{m}{j}h^0(X,L^{\otimes(a(m-j)+1)})=B_{X,L,a,r}. $$
Combining the preceding two displays gives
$$ \sum_{V\in\mathcal S}\deg_A(V)\leq B_{X,L,a,r}. $$

This estimate holds for every finite subset $\mathcal S$. Since
$\deg_A(V)=a^m\deg_L(V)\geq a^m$ for every member, $T_{\infty}^r(X,f)$ is
finite. Applying the estimate to the whole set proves both assertions.
\end{proof}

\begin{proof}[Proof of Theorem~\ref{thm: trace bound}]
Set $S:=T_\infty(f)$, which is finite by
Theorem~\ref{thm: general smooth bound}. If $S=\varnothing$, the assertion is
clear. Otherwise, a common iterate of $f$ makes $S$ totally invariant.
Proposition~\ref{prop: invariant ideal vanishing}, with $Z=S$ and $t=1$,
gives $H^1(X,\mathcal I_S\otimes L)=0$. Therefore
$$ H^0(X,L)\to\bigoplus_{x\in S}L|_x $$
is surjective, and $|S|\leq h^0(X,L)$.
\end{proof}

\begin{remark}\label{rem: invariant divisors}
In codimension one over $\mathbb C$, a polarization-independent cardinality
bound is available. Theorem~\ref{thm: general smooth bound} first makes
$T_{\infty}^1(X,f)$ finite. Choose a common iterate $g$ for which all of its
members are $g^{-1}$-stable. The variety $X$ is $\Q$-factorial and klt, while
$g$ remains int-amplified, so \cite[Theorem~1.1]{Zho21} gives
$$ |T_{\infty}^1(X,f)|\leq n+\rho(X). $$
Here $\rho(X)$ is the Picard number. For example, let
$X=(\PP^1)^n$ and $f=\mu_q^{\times n}$, where
$\mu_q([x:y])=[x^q:y^q]$ for some $q\geq2$. The $2n$ coordinate boundary
divisors belong to $T_{\infty}^1(X,f)$. Since $\rho(X)=n$, Zhong's bound is
attained. By contrast, for $L=\OO_X(1,\ldots,1)$ and $a=1$, we have
$h^0(X,L^{\otimes t})=(t+1)^n$, so Theorem~\ref{thm: general smooth bound}
gives
$$ \sum_{V\in T_{\infty}^1(X,f)}\deg_L(V)\leq B_{X,L,1,1}=\sum_{j=0}^{n-1}(-1)^j\binom{n-1}{j}(n-j+1)^n=\frac{n!(n+3)}{2}. $$
Moreover, if a prime divisor has class $\OO_X(d_1,\ldots,d_n)$, then
$d_i\geq0$ for every $i$, not all $d_i$ vanish, and its $L$-degree is
$(n-1)!\sum_i d_i$. Thus $|T_{\infty}^1(X,f)|\leq n(n+3)/2$, which is weaker
than $2n$ for $n\geq2$. On the other hand,
Theorem~\ref{thm: general smooth bound} controls the sum of $L$-degrees and
applies in every codimension.
\end{remark}

\subsection{Projective space and toric varieties}

\begin{proof}[Proof of Theorem~\ref{thm: projective space bound}]
Apply Theorem~\ref{thm: general smooth bound} with
$X=\PP^n$, $L=\OO_{\PP^n}(1)$, and $a=1$. Here a complete intersection of
$m=n-r$ members of $|L|$ is $\PP^r$, so the proof of that theorem identifies
the constant with
$$ B_{\PP^n,\OO_{\PP^n}(1),1,r}=h^0(\PP^r,\OO_{\PP^r}(m+1))=\binom{n+1}{r}.$$

For $r=n$, after taking a common iterate,
Proposition~\ref{prop: invariant ideal vanishing}, applied to the reduced union
$T_\infty(f)$, gives
$H^1(\PP^n,\mathcal I_{T_\infty(f)}(1))=0$. Thus linear forms restrict
surjectively to $T_\infty(f)$, which is equivalent to projective linear
independence.

For sharpness, consider the coordinate power map $P_q$. Its coordinate linear
subspaces
$$ \{X_{i_1}=\cdots=X_{i_r}=0\},\qquad 0\leq i_1<\cdots<i_r\leq n, $$
are totally invariant and have degree one. There are
$\binom{n+1}{r}$ of them.
\end{proof}

\begin{remark}\label{rem: characteristic hypothesis}
The assumption that $q$ is invertible is essential. When
$k=\overline{\mathbb F}_p$, the coordinate $p$-power map has infinitely many
totally-periodic linear subspaces of every positive codimension: each such
subspace is defined over a finite field and becomes totally invariant under a
suitable iterate.
\end{remark}

\begin{proof}[Proof of Theorem~\ref{thm: toric bound}]
The divisor $D$ is very ample by \cite[Theorem~6.1.15]{CLS11}. Apply
Theorem~\ref{thm: general smooth bound} with $L=\OO_X(D)$ and $a=1$. By
\cite[Proposition~4.3.3]{CLS11},
$$ h^0(X,\OO_X(tD))=|tP_D\cap M| $$
for every $t>0$, which proves the degree bound and its point-case
consequence.
\end{proof}

\section{The totally-periodic polynomial}\label{sec: totally-periodic polynomial}

\begin{proof}[Proof of Proposition~\ref{prop: polynomial functoriality}]
For~(1), a subvariety is totally-periodic under $f^m$ if and only if it is
totally-periodic under $f$. 

For~(2), set $F:=f\circ g$ and $G:=g\circ f$. Throughout this paragraph,
all inverse images and images are set-theoretic and carry the reduced induced
structure. We show that inverse image under $f$ gives a bijection from
$T_{\infty}^r(X,F)$ to $T_{\infty}^r(X,G)$.

Let $V\in T_{\infty}^r(X,F)$ and choose $a\geq1$ such that
$F^{-a}(V)=V$. Since $f\circ G=F\circ f$, we have
$$ G^{-a}(f^{-1}(V))=f^{-1}(F^{-a}(V))=f^{-1}(V). $$
Moreover, $F^a=f\circ G^{a-1}\circ g$, so
$(G^{a-1}\circ g)^{-1}(f^{-1}(V))=V$. Since $G^{a-1}\circ g$ is finite and
surjective,
$$ f^{-1}(V)=(G^{a-1}\circ g)(V), $$
which is irreducible of codimension $r$. Thus
$f^{-1}(V)\in T_{\infty}^r(X,G)$.

Conversely, let $W\in T_{\infty}^r(X,G)$ and choose $a\geq1$ such that
$G^{-a}(W)=W$. The identity $G^a=g\circ F^{a-1}\circ f$ rewrites this as
$$ W=f^{-1}((g\circ F^{a-1})^{-1}(W)) $$
and, after applying $f$, gives
$f(W)=(g\circ F^{a-1})^{-1}(W)$, and hence
$f^{-1}(f(W))=W$. Moreover,
$$ f^{-1}(F^{-a}(f(W)))
=G^{-a}(f^{-1}(f(W)))=W=f^{-1}(f(W)). $$
The surjectivity of $f$ gives $F^{-a}(f(W))=f(W)$. Since $f$ is finite,
$f(W)$ is irreducible of codimension $r$. Thus $W\mapsto f(W)$ is the
inverse bijection.

For~(3), choose ample line bundles $L_X$ and $L_Y$ such that $f^*L_X\otimes L_X^{-1}$ and $g^*L_Y\otimes L_Y^{-1}$ are ample, respectively. The exterior tensor product $L_X\boxtimes L_Y$ is ample, and
$$ (f\times g)^*(L_X\boxtimes L_Y)\otimes(L_X\boxtimes L_Y)^{-1}\simeq(f^*L_X\otimes L_X^{-1})\boxtimes(g^*L_Y\otimes L_Y^{-1}) $$
is ample. Moreover, $\deg(f\times g)=\deg(f)\deg(g)$ is invertible in $k$, so Theorem~\ref{thm: general smooth bound} applies to $f\times g$. For $V\in T_{\infty}^i(X,f)$ and $W\in T_{\infty}^j(Y,g)$, where $1\leq i\leq n$ and $1\leq j\leq m$, a common iterate satisfies
$$ \bigl((f\times g)^a\bigr)^{-1}(V\times W)=V\times W. $$
The same holds when $i=0$ and $V=X$, or when $j=0$ and $W=Y$. The product
$V\times W$ is integral of codimension $i+j$, and its two projections recover
$V$ and $W$. Hence distinct pairs give distinct totally-periodic
subvarieties, and comparison of coefficients proves the inequality.

For~(4), the lower bound follows from the nonnegativity of the coefficients,
and the upper bound follows from
Theorem~\ref{thm: projective space bound}. If a coefficient reaches its
upper bound, the corresponding
degree bound in Theorem~\ref{thm: projective space bound} forces every member
to have degree one, and hence to be linear. The characterization of equality
on the left follows directly from the definition of $\Phi_{\PP^n,f}$.

Suppose that equality holds on the right. Then $T_{\infty}^1(f)$ consists of
$n+1$ hyperplanes, say $H_i=(\ell_i=0)$ for $0\leq i\leq n$. The bijection
established in~(2), applied with the second endomorphism equal to the identity,
shows that reduced inverse image permutes these hyperplanes. Choose $a\geq1$
such that $g:=f^a$ fixes every $H_i$, and put $Q:=q^a$. For a homogeneous
lift $G$ of $g$, set-theoretic invariance yields constants $b_i\in k^*$
satisfying
$$ \ell_i\circ G=b_i\ell_i^Q. $$

We claim that $\ell_0,\ldots,\ell_n$ are linearly independent. Otherwise,
after relabeling a minimally dependent subset, write
$\ell_s=\sum_{j=0}^{s-1}d_j\ell_j$, where the $\ell_j$ on the right are
linearly independent and at least two of the $d_j$ are nonzero. Pulling back
this relation gives
$$ b_s\left(\sum_{j=0}^{s-1}d_j\ell_j\right)^Q=\sum_{j=0}^{s-1}d_jb_j\ell_j^Q. $$
The linear forms on the right may be completed to a coordinate system. For
$d_id_j\neq0$ and $i\neq j$, the coefficient of
$\ell_i^{Q-1}\ell_j$ on the left is $b_sQd_i^{Q-1}d_j\neq0$, since $Q$ is
invertible in $k$, whereas the right-hand side contains only pure $Q$-th
powers. This is a contradiction.

Thus, taking $x_i=\ell_i$, we have
$$ g([x_0:\cdots:x_n])=[b_0x_0^Q:\cdots:b_nx_n^Q]. $$
Choose $\alpha_i\in k^*$ such that $\alpha_i^{Q-1}=b_i^{-1}$, and let
$A([x_0:\cdots:x_n]):=[\alpha_0x_0:\cdots:\alpha_nx_n]$. Then
$A^{-1}gA=P_Q$ (the coordinate power map).

Conversely, if some $f^a$ is projectively conjugate to $P_{q^a}$, then the
conjugates of the coordinate linear subspaces of $P_{q^a}$ attain every
coefficient of the upper bound. Projective conjugacy preserves total
periodicity, and~(1) gives
$\Phi_{\PP^n,f}(t)=\Phi_{\PP^n,f^a}(t)=(1+t)^{n+1}$.

Finally, assume that $\mathrm{char}(k)=0$, and let
$M:=\mathrm{End}_q(\PP^n)$ be the parameter space of degree-$q$
endomorphisms. By Theorem~\ref{thm: projective space bound},
$\deg(V)\leq\binom{n+1}{r}$ for every $V\in T_{\infty}^r(f)$. Moreover,~(2)
shows that reduced inverse image permutes $T_{\infty}^r(f)$. Hence, for every
such $V$, there is an integer $a$ with
$1\leq a\leq\binom{n+1}{r}$ such that $(f^a)^{-1}(V)=V$.

Fix integers $r,e,a$ such that $1\leq r\leq n$ and
$1\leq e,a\leq\binom{n+1}{r}$, and let
$\mathrm{Chow}_{n-r,e}(\PP^n)$ be Koll\'ar's projective seminormal Chow
scheme of effective $(n-r)$-cycles of degree $e$
\cite[Theorem~I.3.21]{Kol96}. Since $M$ is smooth,
$M\times\mathrm{Chow}_{n-r,e}(\PP^n)$ is seminormal. The universal $a$-th
iterate
$$ M\times\PP^n\to M\times\PP^n,\qquad(f,x)\mapsto(f,f^a(x)), $$
is projective with finite fibers, hence finite, and is flat by miracle
flatness. Flat pullback of families of cycles commutes with base change
\cite[Tag~0H4P]{Stacks}. Pulling back the universal cycle after base change
to $M\times\mathrm{Chow}_{n-r,e}(\PP^n)$ therefore gives a family of
effective cycles. By the representing property of this Chow scheme, the
family defines a morphism
$$ \tau_{r,e,a}:M\times\mathrm{Chow}_{n-r,e}(\PP^n)
\to\mathrm{Chow}_{n-r,q^{ar}e}(\PP^n),\qquad
(f,Z)\mapsto(f^a)^*Z. $$
The map $(f,Z)\mapsto q^{ar}Z$ also defines a morphism to the same Chow
scheme. Since the target Chow scheme is separated, the locus
$$ \mathcal C_{r,e,a}:=\{(f,Z):(f^a)^*Z=q^{ar}Z\} $$
where these two morphisms agree is closed. Its projection to $M$ is closed
because $\mathrm{Chow}_{n-r,e}(\PP^n)$ is projective.
Here the displayed condition is an equality of effective cycles. If
$V\in T_{\infty}^r(f)$ has degree $e$ and satisfies
$(f^a)^{-1}(V)=V$, then $(f^a)^*V$ is supported on $V$, and comparison of
degrees gives $(f^a)^*V=q^{ar}V$. Thus $(f,V)\in\mathcal C_{r,e,a}$.

For $r<n$, the image of $\mathcal C_{r,e,a}$ is proper. Indeed, an effective
cycle $Z$ satisfying the displayed equality has
$(f^a)^{-1}(\mathrm{Supp}(Z))=\mathrm{Supp}(Z)$. Put
$W:=\mathrm{Supp}(Z)$. Then $f^a|_W:W\to W$ is finite and surjective and
hence permutes the irreducible components of $W$. After a further iterate,
this permutation is the identity. Since inverse images under a finite
morphism are pure-dimensional, that iterate fixes every component under
reduced inverse image. As $r<n$, these components are positive-dimensional
proper totally-periodic subvarieties. This cannot occur for the generic
endomorphism by \cite[Theorem~1.2]{Fak14}.

For $r=n$, properness can be seen explicitly. Let
$E:=\{e_0,\ldots,e_n\}$ be the set of coordinate vertices, choose
$\gamma\in\mathrm{PGL}_{n+1}(k)$ such that $\gamma(E)\cap E=\varnothing$,
and set $h:=\gamma\circ P_q$, where $P_q$ is the coordinate power map. For a
point $p\in\PP^n$, the set-theoretic inverse image $h^{-1}(p)$ consists of a
single point precisely when $p\in\gamma(E)$, and its unique preimage then
belongs to $E$. If
$(h^b)^{-1}(p)=\{p\}$ for some $b\geq1$, then $p\in\gamma(E)$. For $b=1$,
its unique preimage would also be $p\in E$. If $b>1$, let $p_1$ be the
unique point of $h^{-1}(p)$. Then $p_1\in E$, while
$(h^{b-1})^{-1}(p_1)$ consists of a single point. In particular,
$h^{-1}(p_1)$ consists of a single point, so $p_1\in\gamma(E)$. Both cases contradict
$\gamma(E)\cap E=\varnothing$. Thus $h$ has no totally-periodic point. If
an effective zero-cycle $Z$ satisfied
$(h^a)^*Z=q^{an}Z$, then $h^a$ would permute its finite support, and a
further iterate would fix each point under reduced inverse image, a
contradiction. Hence the images of all $\mathcal C_{n,e,a}$ are proper.

There are only finitely many triples $(r,e,a)$ with $1\leq r\leq n$ and
$1\leq e,a\leq\binom{n+1}{r}$. Hence only finitely many proper closed
subsets of the irreducible parameter space $M$ occur. Their complement is
a nonempty Zariski open subset on which
$\Phi_{\PP^n,f}(t)=1+t^{n+1}$.
\end{proof}

\begin{proposition}\label{prop: polynomial coefficient bounds}
Let $X$ be a smooth projective variety of dimension $n$, $L$ an ample line
bundle on $X$, and $f:X\to X$ an int-amplified endomorphism such that
$\deg(f)$ is invertible in $k$. Fix $a\geq1$ such that $L^{\otimes a}$ is
very ample. Then
$$ 1+t^{n+1}\preceq\Phi_{X,f}(t)\preceq1+\sum_{r=1}^n\left\lfloor\frac{B_{X,L,a,r}}{a^{n-r}}\right\rfloor t^r+t^{n+1}. $$
\end{proposition}

\begin{proof}[Proof of Proposition~\ref{prop: polynomial coefficient bounds}]
The lower bound follows from the nonnegativity of the coefficients, and the
upper bound follows from Theorem~\ref{thm: general smooth bound}.
\end{proof}

\begin{remark}\label{rem: Hirzebruch counterexample}
The equality $\Phi_{X,f}(t)=(1+t)^{\dim X+1}$ does not characterize
projective space when the underlying variety is allowed to vary. Over
$\mathbb C$, consider
$$ h:\PP^2\to\PP^2,\qquad h([x:y:z])=[x^2:y^2:z^2+x^2], $$
and set $p:=[0:0:1]$. We have $h^{-1}(p)=\{p\}$. In the local coordinates
$u=x/z$ and $v=y/z$ at $p$, the pullback of the maximal ideal of $p$ is
$(u^2,v^2)$. Let $\pi:X:=\mathrm{Bl}_p\PP^2\to\PP^2$, with exceptional
curve $E$. The pullback ideal becomes $\OO_X(-2E)$, so the universal property
of the blow-up gives a lift $g:X\to X$. If $H:=\pi^*\OO_{\PP^2}(1)$, then
$$ g^*H=2H,\qquad g^*E=2E. $$
The divisor $2H-E$ is ample and $g^*(2H-E)=2(2H-E)$. Thus $g$ is a polarized
endomorphism of $X\simeq\mathbf F_1$ and, in particular, is finite and
surjective.

We first determine the totally-periodic curves. If an integral curve
$C\subseteq\PP^2$ satisfies $(h^a)^{-1}(C)=C$, then comparison of degrees
gives $(h^a)^*C=2^aC$, so $C$ is a component of $R_{h^a}$. The Jacobian
determinant of $h$ is $8xyz$. Writing $P(w):=w^2+1$, the components of the
iterated ramification divisors, other than $(x=0)$ and $(y=0)$, are the lines
$L_c=(z=cx)$, where either $c=0$ or $P^j(c)=0$ for some $j\geq1$. If $L_c$ were
totally-periodic, then for some $a\geq1$ we would have
$$ P^a(w)-c=(w-c)^{2^a}, $$
where $P^a$ denotes the $a$-th iterate. Comparing the coefficients of
$w^{2^a-1}$ gives $c=0$, whereas the constant terms would then give
$P^a(0)=0$, contradicting $P^a(0)\in\mathbb Z_{>0}$. Thus the only
totally-periodic curves of $h$ are $(x=0)$ and $(y=0)$. Every
totally-periodic curve of $g$ distinct from $E$ descends to one of these two
curves. Conversely, their strict transforms and $E$ are totally invariant
under $g$. Hence $X$ has exactly three
totally-periodic curves.

It remains to count the totally-periodic points. The points $y\in\PP^2$ for
which $h^{-1}(y)$ consists of a single point are
$$ [1:0:1],\qquad [0:1:0],\qquad [0:0:1], $$
and their unique preimages are, respectively,
$$ [1:0:0],\qquad [0:1:0],\qquad [0:0:1]. $$
The last two points are fixed by $h$, whereas the unique preimage
$[1:0:0]$ of $[1:0:1]$ has more than one preimage. Thus the last two are
precisely the totally-periodic points of $h$. The point $[0:1:0]$ remains totally-periodic
on $X$, while $p$ is replaced by $E$ and
$$ g|_E:[u:v]\mapsto[u^2:v^2]. $$
The restriction has exactly two totally-periodic points. Since
$g^{-1}(E)=E$, these and $[0:1:0]$ are all the totally-periodic points of
$g$. Consequently,
$$ \Phi_{X,g}(t)=1+3t+3t^2+t^3=(1+t)^3, $$
but $X\simeq\mathbf F_1$ is not isomorphic to $\PP^2$.
\end{remark}

\section{Further applications}\label{sec: further applications}

We conclude with several consequences of trace splitting, invariant-ideal
vanishing, and the cycle bounds. The external results used below are cited
where they enter.

\subsection{Sharp bound for Completely ramified points}

Let $f:\PP^n\to\PP^n$ satisfy
$f^*\OO_{\PP^n}(1)\simeq\OO_{\PP^n}(q)$ for an integer $q>1$. A closed point
$p\in\PP^n$ is \emph{completely ramified} if $f^{-1}(p)$ consists of one point
set-theoretically. Denote the set of such points by $\mathrm{CR}(f)$.
The following proposition applies the trace splitting from
Lemma~\ref{lem: ideal trace splitting}, allowing the reduced subschemes on the
source and target to be different.

\begin{proposition}\label{prop: completely ramified points}
Assume that $q$ is invertible in $k$ and $q\geq n+1$. Then
the points of $\mathrm{CR}(f)$ are projectively linearly independent. In
particular, $|\mathrm{CR}(f)|\leq n+1$, and this bound is sharp.
\end{proposition}

\begin{proof}
For a finite subset $S\subseteq\mathrm{CR}(f)$, regarded as a reduced closed
subscheme, set $T_S:=(f^{-1}(S))_{\mathrm{red}}$. The map $T_S\to S$ is a
bijection, so $|T_S|=|S|$. $f$ is finite and flat, and therefore locally free of rank $\deg(f)=q^n$. Since
$T_S=(f^{-1}(S))_{\mathrm{red}}$, the same argument as in the proof of
Lemma~\ref{lem: ideal trace splitting}, with $S$ on the target and $T_S$ on the
source, gives a splitting
$$ \mathcal I_S\to f_*\mathcal I_{T_S}
\xrightarrow{q^{-n}\Tr_f}\mathcal I_S. $$
After twisting by $\OO_{\PP^n}(t)$, the projection formula gives a split
injection
$$ H^i(\PP^n,\mathcal I_S(t))\to
H^i(\PP^n,\mathcal I_{T_S}(qt)) $$
for every $i\geq0$ and $t\in\mathbb Z$.

We claim that any finite set $U\subseteq\PP^n$ satisfies
$H^1(\PP^n,\mathcal I_U(d))=0$ for $d\geq|U|-1$. Indeed, products of
$|U|-1$ hyperplanes separate the points of $U$, so the evaluation map is
surjective in degree $|U|-1$. Multiplication by a power of a linear form
nonvanishing on $U$ gives the same conclusion in every degree
$d\geq|U|-1$, and the vanishing follows from the corresponding ideal-sheaf exact sequence.

Suppose that $\mathrm{CR}(f)$ contains a subset $S$ of cardinality $n+2$.
Since $|T_S|=n+2$ and $q\geq n+1$, the preceding fact gives
$H^1(\PP^n,\mathcal I_{T_S}(q))=0$. The split injection with $i=t=1$ then
gives $H^1(\PP^n,\mathcal I_S(1))=0$. Hence linear forms would restrict
surjectively to $S$, contradicting
$h^0(\PP^n,\OO_{\PP^n}(1))=n+1<|S|$. Therefore
$|\mathrm{CR}(f)|\leq n+1$.

Now set $S:=\mathrm{CR}(f)$. Since $q\geq n+1\geq|T_S|-1$, the same argument
gives $H^1(\PP^n,\mathcal I_S(1))=0$. Thus linear forms restrict
surjectively to $S$, which is equivalent to the projective linear
independence of its points.
Finally, the coordinate power map has the $n+1$ coordinate vertices as
completely ramified points, so the bound is sharp.
\end{proof}

\begin{remark}
For $n=2$ over $\mathbb C$, Proposition~\ref{prop: completely ramified points}
gives the sharp bound $|\mathrm{CR}(f)|\leq3$ whenever $q\geq3$.
Amerik--Campana proved the degree-independent bound
$|\mathrm{CR}(f)|\leq9$ \cite[Theorem~1]{AC05}. The trace argument above does
not improve their estimate in the quadratic case.
\end{remark}

\subsection{Regularity, thresholds, and Fano blow-ups}

Throughout this subsection, let $f:\PP^n\to\PP^n$ satisfy
$f^*\OO_{\PP^n}(1)\simeq\OO_{\PP^n}(q)$ for an integer $q>1$ that is
invertible in $k$. Over $\mathbb C$, for a nonzero proper coherent ideal
$\mathfrak a$ on a smooth variety $V$, we write
$\mathrm{lct}(V;\mathfrak a):=
\inf_{\mathrm{ord}_E(\mathfrak a)>0}A_V(E)/\mathrm{ord}_E(\mathfrak a)$,
where $E$ ranges over the prime divisors over $V$ and $A_V(E)$ is the log
discrepancy.

\begin{proposition}\label{prop: invariant regularity}
Let $1\leq r\leq n$, and let $Z\subseteq\PP^n$ be a nonempty reduced closed subscheme
whose irreducible components belong to $T_{\infty}^r(f)$, and set $m:=n-r$.
Then $\mathcal I_Z$ is $(m+2)$-regular. In particular,
$\mathcal I_Z(m+2)$ is globally generated, so $Z$ is scheme-theoretically cut
out by forms of degree $m+2$.

If $k=\mathbb C$, then
$$ \mathrm{lct}(\PP^n;\mathcal I_Z)\geq\frac{r}{m+2}=\frac{r}{n-r+2}. $$
\end{proposition}

\begin{proof}
A common iterate of $f$ makes $Z$ totally invariant. Put $d:=m+2$. For
$1\leq i\leq m+1$, Proposition~\ref{prop: invariant ideal vanishing} gives
$$ H^i(\PP^n,\mathcal I_Z(d-i))=0 $$
because $d-i>0$. If $i\geq m+2$, then
$H^{i-1}(Z,\OO_Z(d-i))=0$ by $\dim Z=m$. It follows that
$$ H^i(\PP^n,\mathcal I_Z(d-i))=H^i(\PP^n,\OO_{\PP^n}(d-i))=0. $$
Thus $\mathcal I_Z$ is $d$-regular,
and $\mathcal I_Z(d)$ is globally generated \cite[Tag~089X]{Stacks}.

For the last assertion, put
$c:=\mathrm{lct}(\PP^n;\mathcal I_Z)$, and let $e$ be the codimension of the
non-klt locus of $(\PP^n,c\cdot Z)$. Since this locus is contained in
$Z$, one has $e\geq r$. The estimate of de Fernex--Ein--Musta\c{t}\u{a}
\cite[Corollary~3.6]{dFEM03}, applied to the degree-$d$ equations of $Z$,
gives $\mathrm{lct}(\PP^n;\mathcal I_Z)\geq e/d\geq r/d$.
\end{proof}

\begin{remark}\label{rem: Rees discrepancy bound}
The preceding threshold estimate has a valuative form. Assume $k=\mathbb C$
and $r\geq2$, and let $\pi:Y\to\PP^n$ be the normalized blow-up of
$\mathcal I_Z$. Write
$$ \mathcal I_Z\OO_Y=\OO_Y\left(-\sum_i b_iE_i\right). $$
The $E_i$ are the Rees divisors of $\mathcal I_Z$, and the valuative
definition of the log canonical threshold gives
$$ A_{\PP^n}(E_i)\geq\frac{r}{n-r+2}b_i $$
for every $i$.
\end{remark}

\begin{corollary}\label{cor: Fano blow-up}
Assume that $n\geq2$, and let $Z\in T_{\infty}^2(f)$ be smooth. Then the blow-up
$\mathrm{Bl}_Z\PP^n$ is Fano. 
\end{corollary}

\begin{proof}
Let $\pi:Y:=\mathrm{Bl}_Z\PP^n\to\PP^n$,
$H:=\pi^*\OO_{\PP^n}(1)$, and $E$ be the exceptional divisor. Since
$\mathcal I_Z\OO_Y=\OO_Y(-E)$, pulling back the evaluation map for the
globally generated sheaf $\mathcal I_Z(n)$ from
Proposition~\ref{prop: invariant regularity} shows that $D:=nH-E$ is globally
generated. It is also $\pi$-ample. If $\varphi_D:Y\to\PP^N$ is the morphism
defined by $D$, then
$(\pi,\varphi_D):Y\to\PP^n\times\PP^N$ is finite: a positive-dimensional
fiber would be contained in a fiber of $\pi$ and contracted by $D$, contrary
to relative ampleness. Hence
$$ \OO_Y(H+D)\simeq
(\pi,\varphi_D)^*\OO_{\PP^n\times\PP^N}(1,1) $$
is ample. The canonical divisor formula
for a smooth codimension-two blow-up gives
$$ -K_Y=-\pi^*K_{\PP^n}-E\sim H+D, $$
so $Y$ is Fano. 
\end{proof}

Set-theoretic total invariance of $Z$ does not by itself imply that $f$ lifts
to the ordinary blow-up. Corollary~\ref{cor: Fano blow-up} asserts that this
blow-up is Fano and does not require such a lift.

\subsection{Log canonical centers of invariant boundaries}

\begin{corollary}\label{cor: lc center bound}
Let $X$ be a smooth complex projective variety of dimension $n$, $L$ an ample
line bundle on $X$, and $f:X\to X$ an int-amplified endomorphism. Suppose
that $D$ is a reduced effective divisor such that $f^{-1}(D)=D$, and choose
an integer $a\geq1$ such that $A:=L^{\otimes a}$ is very ample. Then
$(X,D)$ is log canonical. For
$1\leq r\leq n$, let $\mathrm{LCC}_r(X,D)$ denote the set of its log
canonical centers of codimension $r$. Then
$$ \sum_{W\in\mathrm{LCC}_r(X,D)}\deg_A(W)\leq B_{X,L,a,r}. $$
In particular,
$$ |\mathrm{LCC}_r(X,D)|\leq\left\lfloor\frac{B_{X,L,a,r}}{a^{n-r}}\right\rfloor. $$
For $X=\PP^n$ and $L=\OO_{\PP^n}(1)$, this becomes
$$ \sum_{W\in\mathrm{LCC}_r(\PP^n,D)}\deg(W)\leq\binom{n+1}{r}. $$
\end{corollary}

\begin{proof}
Suppose that $(X,D)$ is not log canonical. By
\cite[Theorem~1.4]{BH14}, after replacing $f$ by an iterate, an irreducible
component $Z$ of the non-log-canonical locus is totally invariant and
$\deg(f|_Z)=\deg(f)$. The morphism $f$ is finite, and since $X$ is smooth it
is flat, so
$f^*[Z]=c[Z]$ for some integer $c>0$. As $f$ is int-amplified and $Z$ is a
proper effective cycle, \cite[Lemma~2.6]{MZ20} gives $c>1$. On the other hand,
the projection formula gives $c\deg(f|_Z)=\deg(f)$, a contradiction. Hence
$(X,D)$ is log canonical.

By \cite[Lemma~2.10]{BH14}, every log canonical center becomes totally
invariant after replacing $f$ by an iterate. Thus the centers of codimension
$r$ belong to $T_{\infty}^r(X,f)$. The first two estimates follow from
Theorem~\ref{thm: general smooth bound}, and the last one follows from
Theorem~\ref{thm: projective space bound}.
\end{proof}

\subsection{Contractible extremal rays}

We use the term \emph{contractible extremal ray} in the sense of
\cite[Definition~4.1]{MZ20}.

\begin{lemma}\label{lem: fixed contraction component}
Let $X$ be a complex projective variety and $F\subseteq X$ an integral
closed subvariety of positive dimension. There are at most $\dim F$
contractible extremal rays whose contraction locus has $F$ as an irreducible
component.
\end{lemma}

\begin{proof}
We follow \cite[Lemma~4.4]{MZ20}, which uses the top self-intersection
polynomial. Put
$$
\begin{aligned}
V_F&:=\mathrm{im}(N^1(X)_{\mathbb C}\to N^1(F)_{\mathbb C}),\\
S_F&:=\{\alpha\in V_F:\alpha^{\dim F}=0\}.
\end{aligned}
$$
This is the zero locus of a nonzero homogeneous polynomial of degree
$\dim F$, since the restriction of an ample class has positive top
self-intersection on $F$.

For each ray $R$ under consideration, choose a curve $C_R\subseteq F$
spanning $R$ and set
$$ L_R:=\{D|_F\in V_F:D\cdot C_R=0\}. $$
Let $\pi_R:X\to Y_R$ be the contraction of $R$. Such a curve exists because
the curves contracted by $\pi_R$ cover a dense
subset of $F$, and a contracted curve through a general point of $F$ is
contained in $F$. Intersection with $C_R$ defines a well-defined linear
functional on $V_F$; it is nonzero because an ample divisor has positive
intersection with $C_R$. Thus $L_R$ is a hyperplane.

The same covering property shows that a general fiber of
$\pi_R|_F$ is positive-dimensional, so $\dim\pi_R(F)<\dim F$. By
\cite[Definition~4.1(3)]{MZ20}, after complexification, every class in $L_R$
is the restriction of $\pi_R^*P$ for some
$P\in N^1(Y_R)_{\mathbb C}$. The projection formula gives
$$ \bigl((\pi_R^*P)|_F\bigr)^{\dim F}
=P^{\dim F}\cdot(\pi_R)_*[F]=0. $$
Hence $L_R\subseteq S_F$, and $L_R$ is a hyperplane component of $S_F$.

The assignment $R\mapsto L_R$ is injective. Indeed, suppose that
$L_R=L_{R'}$, choose a curve $C_{R'}\subseteq F$ spanning $R'$, and let
$H_R$ be an ample divisor on $Y_R$. Since
$(\pi_R^*H_R)|_F\in L_R=L_{R'}$, we have
$\pi_R^*H_R\cdot C_{R'}=0$. Thus $\pi_R$ contracts $C_{R'}$, and
\cite[Definition~4.1(2)]{MZ20} gives $R'=R$.
A nonzero polynomial of degree $\dim F$ has at most $\dim F$ distinct
hyperplane factors, proving the assertion.
\end{proof}

\begin{proposition}\label{prop: contractible ray bound}
Let $X$ be a smooth complex projective variety of dimension $n$, $L$ an ample
line bundle on $X$, and $f:X\to X$ an int-amplified endomorphism. Fix
$a\geq1$ such that $L^{\otimes a}$ is very ample, and set
$$ b_r:=\left\lfloor\frac{B_{X,L,a,r}}{a^{n-r}}\right\rfloor\qquad(1\leq r\leq n-1). $$
If $\mathcal R_{\mathrm{contr}}(X)$ is the set of contractible extremal rays
of $\overline{\mathrm{NE}}(X)$, then
$$ |\mathcal R_{\mathrm{contr}}(X)|\leq n+\sum_{r=1}^{n-1}(n-r)b_r. $$
Here $b_1$ may be replaced by
$\min\{b_1,n+\rho(X)\}$, where $\rho(X)$ is the Picard number.
If $X$ is a Mori dream space, the same formula bounds all extremal rays of
$\overline{\mathrm{NE}}(X)$. This applies, in particular, when $X$ is
rationally connected.
\end{proposition}

\begin{proof}
For a contractible ray $R$, let $\Sigma_R$ be its contraction locus.
Meng--Zhang \cite[Lemma~4.3]{MZ20} show that every irreducible component of
$\Sigma_R$ is totally-periodic under $f$. The rays with $\Sigma_R=X$ are at
most $n$ by Lemma~\ref{lem: fixed contraction component}. If
$\Sigma_R\neq X$, choose one of its irreducible components $F$. It has
positive dimension, and if $\mathrm{codim}_X(F)=r$, then
$F\in T_{\infty}^r(X,f)$. Lemma~\ref{lem: fixed contraction component} gives at
most $n-r$ rays for each such $F$. Hence
$$ |\mathcal R_{\mathrm{contr}}(X)|\leq n+\sum_{r=1}^{n-1}(n-r)|T_{\infty}^r(X,f)|, $$
and Theorem~\ref{thm: general smooth bound} gives the stated estimate.
The improvement in codimension one follows from
Remark~\ref{rem: invariant divisors}.

On a Mori dream space the nef cone is rational polyhedral and generated by
semiample classes \cite[Definition~1.10 and Proposition~1.11]{HK00}. Let $R$
be an extremal ray of $\overline{\mathrm{NE}}(X)$. Its dual face
$$ R^\perp\cap\mathrm{Nef}(X) $$
is a facet. Choose a rational divisor class in its relative interior and take
the Stein factorization $\pi:X\to Y$ of the morphism defined by a semiample
multiple. Then $\pi_*\OO_X=\OO_Y$, and a curve is contracted by $\pi$ if and
only if its numerical class lies in $R$. Moreover, the contraction has
relative Picard number one. Hence $\pi^*N^1(Y)_{\mathbb Q}$ and
$R^\perp\subseteq N^1(X)_{\mathbb Q}$ are codimension-one subspaces, and the
projection formula gives
$\pi^*N^1(Y)_{\mathbb Q}\subseteq R^\perp$. They are therefore equal, which
verifies \cite[Definition~4.1(3)]{MZ20}. Thus every extremal ray is
contractible.
Finally, a smooth rationally connected projective variety admitting an
int-amplified endomorphism is of Fano type
\cite[Corollary~6.7]{Yos21}, hence is a Mori dream space by
\cite[Corollary~1.3.2]{BCHM10}.
\end{proof}

\begin{remark}\label{rem: quantitative MZ}
Proposition~\ref{prop: contractible ray bound} is a quantitative refinement
of the finiteness theorem of Meng--Zhang
\cite[Theorem~4.5]{MZ20}. The resulting bound depends linearly on the numbers
$b_r$ because the proof fixes one irreducible component of a contraction
locus at a time, rather than counting all possible unions of such components.
\end{remark}


\begin{thebibliography}{99}

\bibitem{AC05}
E.~Amerik and F.~Campana,
\emph{Exceptional points of an endomorphism of the projective plane},
Math. Z. \textbf{249} (2005), no.~4, 741--754.

\bibitem{BCHM10}
C.~Birkar, P.~Cascini, C.~D.~Hacon, and J.~McKernan,
\emph{Existence of minimal models for varieties of log general type},
J. Amer. Math. Soc. \textbf{23} (2010), no.~2, 405--468.

\bibitem{BH14}
A.~Broustet and A.~H\"oring,
\emph{Singularities of varieties admitting an endomorphism},
Math. Ann. \textbf{360} (2014), no.~1--2, 439--456.

\bibitem{CLN00}
D.~Cerveau and A.~Lins Neto,
\emph{Hypersurfaces exceptionnelles des endomorphismes de
$\mathbb{C}\mathbf{P}(n)$},
Bol. Soc. Brasil. Mat. (N.S.) \textbf{31} (2000), no.~2, 155--161.

\bibitem{CLS11}
D.~A.~Cox, J.~B.~Little, and H.~K.~Schenck,
\emph{Toric varieties},
Graduate Studies in Mathematics, vol.~124, American Mathematical Society,
Providence, RI, 2011.

\bibitem{dFEM03}
T.~de Fernex, L.~Ein, and M.~Musta\c{t}\u{a},
\emph{Bounds for log canonical thresholds with applications to birational
rigidity},
Math. Res. Lett. \textbf{10} (2003), no.~2--3, 219--236.

\bibitem{DS03}
T.-C.~Dinh and N.~Sibony,
\emph{Dynamique des applications d'allure polynomiale},
J. Math. Pures Appl. (9) \textbf{82} (2003), no.~4, 367--423.

\bibitem{DS08}
T.-C.~Dinh and N.~Sibony,
\emph{Equidistribution towards the Green current for holomorphic maps},
Ann. Sci. \'Ec. Norm. Sup\'er. (4) \textbf{41} (2008), no.~2, 307--336.

\bibitem{DS09}
T.-C.~Dinh and N.~Sibony,
\emph{Super-potentials of positive closed currents, intersection theory and
dynamics},
Acta Math. \textbf{203} (2009), no.~1, 1--82.

\bibitem{DS10}
T.-C.~Dinh and N.~Sibony,
\emph{Equidistribution speed for endomorphisms of projective spaces},
Math. Ann. \textbf{347} (2010), no.~3, 613--626.

\bibitem{Fak03}
N.~Fakhruddin,
\emph{Questions on self maps of algebraic varieties},
J. Ramanujan Math. Soc. \textbf{18} (2003), no.~2, 109--122.

\bibitem{Fak14}
N.~Fakhruddin,
\emph{The algebraic dynamics of generic endomorphisms of $\mathbf P^n$},
Algebra Number Theory \textbf{8} (2014), no.~3, 587--608.

\bibitem{Fav03}
C.~Favre,
\emph{Equidistribution problems in holomorphic dynamics in $\mathbf P^2$},
in \emph{Dynamical systems. Part II}, Publ. Cent. Ric. Mat. Ennio Giorgi,
Scuola Norm. Sup., Pisa, 2003, pp.~79--111.

\bibitem{FJ03}
C.~Favre and M.~Jonsson,
\emph{Brolin's theorem for curves in two complex dimensions},
Ann. Inst. Fourier (Grenoble) \textbf{53} (2003), no.~5, 1461--1501.

\bibitem{Fuj83}
T.~Fujita,
\emph{Vanishing theorems for semipositive line bundles},
in \emph{Algebraic geometry (Tokyo/Kyoto, 1982)},
Lecture Notes in Mathematics, vol.~1016, Springer, Berlin, 1983,
pp.~519--528.

\bibitem{Ful}
W.~Fulton,
\emph{Intersection theory},
2nd ed., Ergebnisse der Mathematik und ihrer Grenzgebiete (3), vol.~2,
Springer-Verlag, Berlin, 1998.


\bibitem{FS94}
J.~E.~Forn\ae ss and N.~Sibony,
\emph{Complex dynamics in higher dimension. I},
Ast\'erisque \textbf{222} (1994), 201--231.

\bibitem{Har}
R.~Hartshorne,
\emph{Algebraic geometry},
Graduate Texts in Mathematics, vol.~52,
Springer-Verlag, New York--Heidelberg, 1977.

\bibitem{Hor17}
A.~H\"oring,
\emph{Totally invariant divisors of endomorphisms of projective spaces},
Manuscripta Math. \textbf{153} (2017), no.~1--2, 173--182.

\bibitem{HK00}
Y.~Hu and S.~Keel,
\emph{Mori dream spaces and GIT},
Michigan Math. J. \textbf{48} (2000), 331--348.

\bibitem{KT25}
T.~Kawakami and B.~Totaro,
\emph{Endomorphisms of varieties and Bott vanishing},
J. Algebraic Geom. \textbf{34} (2025), no.~2, 381--405.

\bibitem{Kee03}
D.~S.~Keeler,
\emph{Ample filters of invertible sheaves},
J. Algebra \textbf{259} (2003), no.~1, 243--283;
corrigendum, J. Algebra \textbf{507} (2018), 592--598.

\bibitem{Kol96}
J.~Koll\'ar,
\emph{Rational curves on algebraic varieties},
Ergebnisse der Mathematik und ihrer Grenzgebiete (3), vol.~32,
Springer-Verlag, Berlin, 1996.

\bibitem{Laz}
R.~Lazarsfeld,
\emph{Positivity in algebraic geometry. I: Classical setting: line bundles and linear series},
Ergebnisse der Mathematik und ihrer Grenzgebiete (3), vol.~48,
Springer-Verlag, Berlin, 2004.

\bibitem{Mab23}
Y.~Mabed,
\emph{Totally invariant divisors of non trivial endomorphisms of the
projective space},
Geom. Dedicata \textbf{217} (2023), no.~5, Paper No.~79, 12 pp.

\bibitem{Men20}
S.~Meng,
\emph{Building blocks of amplified endomorphisms of normal projective varieties},
Math. Z. \textbf{294} (2020), no.~3--4, 1727--1747.

\bibitem{MZ20}
S.~Meng and D.-Q.~Zhang,
\emph{Semi-group structure of all endomorphisms of a projective variety
admitting a polarized endomorphism},
Math. Res. Lett. \textbf{27} (2020), no.~2, 523--549.

\bibitem{Mil06}
J.~Milnor,
\emph{Dynamics in one complex variable},
3rd ed., Annals of Mathematics Studies, vol.~160, Princeton University Press,
Princeton, NJ, 2006.

\bibitem{Mum74}
D.~Mumford,
\emph{Abelian varieties},
2nd ed., with appendices by C.~P.~Ramanujam and Y.~Manin,
Tata Institute of Fundamental Research Studies in Mathematics, vol.~5,
Oxford University Press, London, 1974.

\bibitem{NZ10}
N.~Nakayama and D.-Q.~Zhang,
\emph{Polarized endomorphisms of complex normal varieties},
Math. Ann. \textbf{346} (2010), no.~4, 991--1018.

\bibitem{Stacks}
The Stacks Project Authors,
\emph{The Stacks Project},
\url{https://stacks.math.columbia.edu}.

\bibitem{Yos21}
S.~Yoshikawa,
\emph{Structure of Fano fibrations of varieties admitting an int-amplified
endomorphism},
Adv. Math. \textbf{391} (2021), Paper No.~107964.

\bibitem{Zha14}
D.-Q.~Zhang,
\emph{Invariant hypersurfaces of endomorphisms of projective varieties},
Adv. Math. \textbf{252} (2014), 185--203.

\bibitem{Zho21}
G.~Zhong,
\emph{Totally invariant divisors of int-amplified endomorphisms of normal
projective varieties},
J. Geom. Anal. \textbf{31} (2021), no.~3, 2568--2593.

\end{thebibliography}
\end{document}